\documentclass[10pt]{amsart}

\usepackage{amssymb,amsmath,amsthm,amsfonts}

\newcommand{\comment}[1]{}

\newtheorem{lem}{Lemma}
\newtheorem{propn}[lem]{Proposition}
\newtheorem{cor}[lem]{Corollary}
\newtheorem{thm}[lem]{Theorem}
\newtheorem{claim}{Claim}
\newtheorem*{thmA}{Theorem A}
\newtheorem*{thmB}{Theorem B}
\theoremstyle{remark}
\theoremstyle{definition}

\newcommand{\Z}{\mathbb Z}
\newcommand{\T}{\mathbb T}

\newcommand{\C}{\mathbb C}

\DeclareMathOperator{\degree}{deg}

\newcommand{\D}{\delta}
\newcommand{\E}{\epsilon}
\newcommand{\VE}{\varepsilon}
\newcommand{\A}{\alpha}
\newcommand{\B}{\beta}

\newcommand{\be}{\begin{equation}}
\newcommand{\ee}{\end{equation}}
\newcommand{\bee}{\begin{equation*}}
\newcommand{\eee}{\end{equation*}}

\begin{document}
\title{Polynomial configurations in difference sets\\ (Revised Version)}
\author{Neil Lyall\quad\quad\quad\'Akos Magyar}
\thanks{Both authors were partially supported by NSF grants.}

\address{Department of Mathematics, The University of Georgia, Boyd
  Graduate Studies Research Center, Athens, GA 30602, USA}
\email{lyall@math.uga.edu}
\address{Department of Mathematics, The University of Georgia, Boyd
  Graduate Studies Research Center, Athens, GA 30602, USA}
\email{magyar@math.uga.edu}

\keywords{Difference sets, S\'ark\"ozy's theorem, polynomial configurations}

\begin{abstract} 
We prove a quantitative version of the Polynomial Szemer\'edi Theorem for \emph{difference sets}.
This result is achieved by first establishing a higher dimensional analogue of a theorem of S\'ark\"ozy (the simplest non-trivial case of the Polynomial Szemer\'edi Theorem asserting that the difference set of any subset of the integers of positive upper density necessarily contains a perfect square) and then applying a simple lifting argument. 

\end{abstract}
\maketitle
{\small
\emph{An earlier version of this article, in which the main results were restricted to the case of linearly independent polynomials, has already appeared in \cite{NL}. The main substantive changes in this revision occur in the statements of Theorem \ref{1} and Corollary \ref{cor} which are obtained through a simple modification of the original argument in Section \ref{2.1}.}}

\setlength{\parskip}{3pt}

\section{Introduction}

\subsection{Background}
A striking and elegant result in density Ramsey theory states that in any subset of the integers of positive upper density there necessarily exist two distinct elements whose difference is a perfect square. 

This result was originally conjectured by L. Lov\'asz and eventually verified independently by Furstenberg \cite{Furst}, using techniques from ergodic theory, and S\'ark\"ozy \cite{Sarkozy}, using an approach similar in spirit to Roth's Fourier analytic (circle method inspired) proof of Szemer\'edi's theorem for arithmetic progressions of length three.

S\'ark\"ozy actually obtained the following stronger quantitative result
.

\begin{thmA}[S\'ark\"ozy \cite{Sarkozy}]
If $A\subseteq[1,N]$ and $d^2\notin A-A$ for any $d\ne0$, then there exists an absolute constant $C>0$ such that
\[\frac{|A|}{N}\leq C\left(\frac{(\log\log N)^2}{\log N}\right)^{1/3}.\]
\end{thmA}

\emph{Notation:} In the theorem above and in the sequel we will use $N$ (and later also $M$) to denote an arbitrary positive integer, $[1,N]$ to denote $\{1,\dots,N\}$ as is customary, and $A\pm A$ to denote the usual difference and sum sets of $A$, namely $A\pm A=\{a\pm a'\,|\,a,a'\in A\}$. 

The current best known quantitative bound of $(\log N)^{-c\log\log\log\log N}$ in Theorem A is due to Pintz, Steiger and Szemer\'edi \cite{PSS}. These methods were later extended by Balog, Pelik\'an, Pintz and Szemer\'edi \cite{BPPS} to obtain the same bounds, with the implicit constant now depending on $k$, for sets with no $k$th power differences. 

We note that it is conjectured that for any $\E>0$ and $N\geq N_0(\E)$ sufficiently large there exists a set $A\subseteq[1,N]$ with 
$|A|\geq N^{1-\E}$ that contains no square differences, see for example \cite{Green}. Ruzsa \cite{Ruzsa} has demonstrated that this is at least true for $\E=0.267$.  

Bergelson and Leibman (extending on the ideas of Furstenberg) established a far reaching  qualitative generalization of 
S\'ark\"ozy's theorem, the so-called Polynomial Szemer\'edi Theorem.

\begin{thmB}[Bergelson and Leibman \cite{BL}] 
If $A$ is a subset of the integers of positive upper density and $P_1(d),\dots,P_\ell(d)$ are polynomials in $\Z[d]$ with $P_i(0)=0$ for $i=1,\dots,\ell $, then there exists $a\in A$ and $d\ne0$ such that
\[a+\{P_1(d),\dots,P_\ell(d)\}\subseteq A.\]
\end{thmB}
We note that no quantitative version of this multiple recurrence result is known beyond the linear case of Szemer\'edi's theorem, see Gowers \cite{Gowers}, and the general single recurrence case of S\'ark\"ozy's theorem (the case $\ell=2$ above).

The purpose of this paper is to establish a quantitative result on the existence of polynomial configurations of the type in Theorem B in the \emph{difference set} of sparse subsets of the integers.

\subsection{Statement of Main Results}

We now fix a family of polynomials 
\[P_1(d),\dots,P_\ell(d)\in\Z[d]\] with $P_i(0)=0$ for $i=1,\dots,\ell $ and set $k=\max_i \degree P_i.$ We will assume throughout this paper that $k\geq2$.

\begin{thm}\label{1} 
If $A\subseteq[1,N]$ and 
$\{P_1(d),\dots ,P_\ell (d)\}\nsubseteq A-A$
for any $d\ne0$,
then we necessarily have
\[\frac{|A|}{N}\leq C\left(\frac{\log\log N}{\log N}\right)^{1/\ell(k-1)}\]
for some absolute constant $C=C(P_1,\dots,P_\ell )$.
\end{thm}
In the case of a single polynomial ($\ell=1$), this result has also recently been obtained by Lucier \cite{Lucier} and, to the best of our knowledge, constitutes the best bounds that are currently known for arbitrary polynomials with integer coefficients and zero constant term.

\comment{
We fix a parameter $0<\D\leq 1$, a family of linearly independent integer polynomials $P_1,\dots,P_\ell $ such that $P_i(0)=0$ for all $i=1,\dots,\ell $ and set $k=\max_i \degree P_i$. 

\begin{thm}\label{1} 
There exists $C=C(P_1,\dots,P_\ell )$ such that if $N\geq \exp(C\D^{-\ell (k-1)}\log\D^{-1})$ and $A\subseteq[1,N]$ with $|A|\geq\D N$, then there exists $d\ne0$ such that 
\[\{P_1(d),\dots ,P_\ell (d)\}\subseteq A-A.\]
\end{thm}

We remark that by \emph{symmetrizing} $A$ one can immediately deduce, from Theorem \ref{1}, the following result on the existence of the same polynomial configurations in a shift of the sumset of $A$.

\begin{cor}
There exists $C'=C'(P_1,\dots,P_\ell )$ such that if $N\geq \exp(C'\D^{-2\ell(k-1)}(\log\D^{-1})$ and $A\subseteq[1,N]$ with $|A|\geq\D N$, then there exists $d\ne0$ and $m\in[1,N]$ such that 
\[2m+\{P_1(d),\dots,P_\ell (d)\}\subseteq A+A.\]
\end{cor}
}

We remark that one can immediately deduce, from Theorem \ref{1}, the following result on the existence of polynomial configurations in (a shift of) the sumset $A+B$ of two given sets $A,B\subseteq[1,N]$.

\begin{cor}\label{cor}
If $A,B\subseteq[1,N]$ and 
$m+\{P_1(d),\dots ,P_\ell (d)\}\nsubseteq A+B$
for any $d\ne0$ and $m\in[2,2N]$,
then we necessarily have
\[\frac{|A||B|}{N^2}\leq C'\left(\frac{\log\log N}{\log N}\right)^{1/\ell(k-1)}\]
for some absolute constant $C'=C'(P_1,\dots,P_\ell )$.
\end{cor}

\begin{proof} Since
\[\sum_{2\leq m\leq2N}\left|B\cap(m-A)\right|=|A||B|\]
it follows that there exists $m\in[2,2N]$ such that if we set $D=B\cap(m-A)$, then
\[|D|\geq \frac{|A||B|}{2N-1}.\]
The result then follows from Theorem \ref{1} since
\[D-D\subseteq  A+B-m
.\qedhere\]
\end{proof}


The strategy we will employ to prove Theorem \ref{1} is to \emph{lift} the problem to $\Z^k$ in such a way that we may then apply the following higher dimensional analogue of S\'ark\"ozy's theorem.

\begin{thm}\label{2} 
If $B\subseteq[1,N]^k$ and $(d,d^2,\dots,d^k)\notin B-B$ for any $d\ne0$ then we necessarily have
\[\frac{|B|}{N^k}\leq C\left(\frac{\log\log N}{\log N}\right)^{1/(k-1)}\]
for some absolute constant $C=C(k)$.
\end{thm}

\comment{
\begin{thm}\label{2} 
There exists $C=C(k)$ such that if $N\geq \exp(C\D^{-(k-1)}(\log\D^{-1})^2)$ and $B\subseteq[1,N]^k$ with $|B|\geq\D N^k$, then there exists $d\ne0$ such that 
\[(d,d^2,\dots,d^k)\in B-B.\]
\end{thm}
}

Since Theorem \ref{2} is concerned with the intersection of a difference set with the \emph{monomial} curve $(d,d^2,\dots,d^k)$ we speculate that the methodology of Balog et al. \cite{BPPS} may be applied in this higher dimensional situation to obtain far superior bounds in Theorem \ref{2} and hence also in Theorem \ref{1}. 

\emph{Further notational convention:} Throughout this paper the letters $c$ and $C$ will denote absolute constants that will generally satisfy $0<c\ll 1\ll C$, whose values may change from line to line and even from step to step, and will unless otherwise specified depend only on the dimension $k$.

\section{Reduction to  the key dichotomy proposition}

We first present the \emph{lifting} argument that allows us to deduce Theorem \ref{1} from Theorem \ref{2}.

\subsection{Proof that Theorem \ref{2} implies Theorem \ref{1}}\label{2.1}

Let $P_i(d)=c_{i1} d+\cdots+c_{ik}d^k$ for $1\leq i\leq\ell$. 

Suppose that the coefficient matrix $\mathcal{P}=\{c_{ij}\}$ has rank $r$ with $1\leq r\leq \ell$. Without loss in generality we will make the additional assumption that it is in fact the first $r$ polynomials $P_1,\dots,P_{r}$ that are linearly independent and use $\mathcal{R}$ to denote the $r\times k$ matrix corresponding to the first $r$ rows of $\mathcal{P}$. 

As a consequence of this assumption it  follows that the remaining polynomials, $P_{r+i}$ with $1\leq i\leq \ell-r$, can be expressed as
\[P_{r+i}=d_{i1}P_1+\cdots+d_{ir} P_{r}\]
where $\mathcal{D}=\{d_{ij}\}$ is some $(\ell-r)\times r$ matrix with rational coefficients.

Note that 
\begin{align*}
&\mathcal{P}:\Z^k\rightarrow\Z^\ell \\
&\mathcal{R}:\Z^k\rightarrow\Z^r \\
&\mathcal{D}:\mathcal{R}(\Z^k)\rightarrow\Z^{\ell-r}
\end{align*} 
and
\[\mathcal{P}(b)=\left(\begin{matrix} \mathcal{R}(b) \\ \mathcal{D}(\mathcal{R}(b))\end{matrix}\right).\]

Let $A^\ell=A\times\cdots\times A\subseteq[1,N]^\ell$ and $\D=|A|/N$.

The full rank assumption on the matrix $\mathcal{R}$ ensures that there exists an absolute constant $c$, depending only on the coefficients of the matrix $\mathcal{R}$, such that
\[\bigl|\mathcal{R}(\Z^k)\cap(A^r-s)\bigr|\geq c \D^r N^r\]
for some $s\in[1,c^{-1}]^r$.
Thus, if we choose $N'$ to be a large enough multiple of $N$ (again depending only the coefficients of the matrix $\mathcal{R}$) and let
\[B'=\left\{b\in[-N',N']^k\,:\,\mathcal{R}(b)\in A^r -s\right\},\]
it follows that 
\[|B'|\geq c\, \D^r  N^k.\]

Since \[\sum_{t\in\Z^{\ell-r}}\sum_{b\in B'}1_{A^{\ell-r}}(\mathcal{D}(\mathcal{R}(b))+t)=|A|^{\ell-r}|B'|\]
it follows that there exists $c=c(\mathcal{P})$ and $t\in\Z^{\ell-r}$ such that 
\[\left|\left\{b\in B'\,:\, \mathcal{D}(\mathcal{R}(b))\in A^{\ell-r}-t\right\}\right|\geq c\delta^{\ell-r}|B'|.\]

Hence, if we let 
\[B=\left\{b\in[-N',N']^k\,:\,\mathcal{P}(b)\in A^\ell -m\right\},\]
where $m=(s,t)\in\Z^\ell$,
it follows that 
\[|B|\geq c\, \D^\ell  N^k.\]

The result now follows from Theorem \ref{2} since if there were to exist a $d\ne0$ such that
\[(d,d^2,\dots,d^k)\in B-B\]
this would immediately implies that
\[(P_1(d),\dots,P_\ell (d))\in A^\ell-A^\ell,\]
since $\mathcal{P}(B)\subseteq A^\ell -m$.\qed

Matters therefore reduce to proving Theorem \ref{2}. 

\subsection{Dichotomy between randomness and arithmetic structure}


Our approach will be to establish a dichotomy between randomness and structure of the following form. 


Let us fix the notation $Q_M=[1,M]\times\cdots\times[1,M^k]$ and $\VE=(10k)^{-1}$.

\begin{propn}
\label{KeyPropn}
Let $B\subseteq Q_M$, $\D=|B|/|Q_M|$, and $\sigma=c_k\D^{k-1}$. If $M\geq \D^{-C}$, with $C>0$ sufficiently large (depending only on $k$), then either
$B$ behaves as though it were a random set in the sense that
\begin{equation}\label{random}
\sum_{d=1}^M \left|B\cap\left(B+(d,d^2,\dots,d^k)\right)\right|\geq \frac{\VE}{4}\,\D|B|M
\end{equation}
or
$B$ has arithmetic structure in the sense that there exists a grid $\Lambda\subseteq Q_M$ of the form
\begin{equation}\label{grid}
\Lambda=\{m+(\ell_1 q,\dots,\ell_k q^k)\,|\,(\ell_1,\dots,\ell_k)\in Q_L\}
\end{equation}
with $L\geq \D^{k+2}\sigma M$ such that
 \[|B\cap \Lambda|>\D(1+\sigma)|\Lambda|.\]
 
\end{propn}

In contrast with the standard $L^\infty$ increment strategy of Roth, we will obtain the dichotomy in Proposition \ref{KeyPropn} by exploiting the concentration of the $L^2$ mass of the Fourier transform. Similar arguments of this type can be found in Heath-Brown \cite{HB} and Szemer\'edi \cite{SzRoth}, see also Ruzsa and Sanders \cite{RandS}.
The proof of Proposition \ref{KeyPropn} will be presented in Sections \ref{stage} and \ref{proof}.


\subsection{Proof that Proposition \ref{KeyPropn} implies Theorem \ref{2}}

It is easy to see, by partitioning $[1,N]^k$ into boxes of size $M\times M^2\times\cdots\times M^k$ with $M$ essentially equal to $N^{1/k}$, that we may, with no loss in generality, assume that $B\subseteq Q_M$ with $\D=|B|/|Q_M|\geq|B|/N^k$.

If $(d,d^2,\dots,d^k)\notin B-B$ for any $d\ne0$ (as is the assumption in Theorem \ref{2}), then Proposition \ref{KeyPropn} allows us to perform an iteration. At the $n$th step of this iteration we will have a set $B_n\subseteq Q_{M_n}$ of size $\D_n|Q_{M_n}|$, this set will be an appropriately rescaled version  of a subset of $B$ itself and hence will also contain no non-trivial differences of the form $(d,d^2,\dots,d^k)$. 

Let $B_0=B$, $M_0=M$ and $\D_0=\D$. Proposition \ref{KeyPropn} ensures that either 
\be\label{3}
M_n\leq \D_n^{-C}
\ee
or else the iteration proceeds allowing us to choose $M_{n+1}$, $\D_{n+1}$ and $B_{n+1}$ such that
\[M_{n+1}\geq c\D_n^{(2k+1)}M_n\]
and
\[\D_{n+1}\geq \D_n+c\D_n^{k}.\]

Now as long as the iteration continues we must have $\D_n\leq1$, and so after $O(\D^{1-k})$ iterations condition (\ref{3}) must be satisfied, giving
\[(\D^{-(2k+1)})^{-C\D^{1-k}}M\leq \D^{-C}.\]
From this it follows that
\[\log M\leq C\D^{-(k-1)}\log \D^{-1}\]
and consequently (after a short calculation that we leave to the reader) that
\[\D\leq C\left(\frac{\log\log M}{\log M}\right)^{1/(k-1)}.\]
This establishes Theorem \ref{2}.\qed

The rest of this article is devoted to the proof of Proposition \ref{KeyPropn}.

\section{Setting the stage for the proof of Proposition \ref{KeyPropn}}\label{stage}

We suppose that $B\subseteq Q_M$, $\D=|B|/ |Q_M|$, and $M\geq \D^{-C}$.
Our approach will be to assume that $B$ exhibits neither of the two properties described in Proposition \ref{KeyPropn} and then seek a contradiction.


\subsection{A simple consequence of $B$ being non-random}
If we were to suppose that $B$ is non-random, in the sense that inequality (\ref{random}) does not hold, then it would immediately follows that 
\be\label{random2}
\sum_{m,n\in\Z^k}1_B(m)1_B(n)1_S(m-n) 
\leq \frac{1}{4}\,\D|B||S|\ee
where
\[S=\{(d,d^2,\dots,d^k)\,:\,1\leq d \leq \VE M\}.\]

\subsection{A simple consequences of $B$ being non-structured}
If we were to assume that $B$ is \emph{regular}, in the sense that $B$ in fact satisfies the inequality
\[|B\cap \Lambda|\leq\D(1+\sigma)|\Lambda|\]
for all arithmetic grids $\Lambda\subseteq Q_M$ of the form (\ref{grid}) with $L\geq\D^{k+2}\sigma M$, then the set
\[B'=B\cap\bigl((\VE M,(1-\VE) M]\times\cdots\times(\VE M^k,(1-\VE)M^k]\bigr)\]
must contain most of the elements of $B$. In particular we must have
\be\label{B'}
|B'|\geq(3/4)|B|\ee
since if this were not the case we would immediately obtain a grid $\Lambda\subseteq Q_M$ of the form (\ref{grid}) with $q=1$ and $L\geq\VE M$ such that
\[|B\cap \Lambda|\geq\D(1+1/4)|\Lambda|.\]

\subsection{The balance function}
We define the \emph{balance function of $B$} to be \[f_B=1_B-\D 1_{Q_M},\]
and note that $f_B$ has mean value zero, that is
$\sum f_B(m)=0.$
This property of the balance function $f_B$ will be critically important in our later arguments. 

It easy to verify that if $B$ satisfies inequalities (\ref{random2}) and (\ref{B'}), then
\be\label{fB1}
\sum_{m,n\in\Z^k}f_B(m)f_B(n)1_S(m-n) \leq -\frac{1}{4}\,\D|B||S|.
\ee
One can see this by simply expanding the sum into a sum of four sums, one involving only the function $1_B$ on which we can apply (\ref{random2}), two involving the functions $1_B$ and $-\D1_{Q_M}$ on which we can apply (\ref{B'}), and one involving only the function 
$-\D1_{Q_M}$ which can be estimated trivially.

\subsection{Fourier analysis on $\Z^k$}
For $f:\Z^k\rightarrow\C$ with finite support we define the \emph{Fourier transform of $f$} to be
\[\widehat{f}(\A)= \sum\limits_{m\in\Z^k}f(m)e^{-2\pi i m\cdot\A}.\]
The finite support assumption on $f$ ensures that $\widehat{f}$ is a continuous function on $\T^k$ and that orthogonality immediately gives both the Fourier inversion formula and Plancherel's identity, namely
\[f(m)=\int_{\T^k}\widehat{f}(\A)e^{2\pi i m\cdot\A}d\A
\quad\quad\text{and}\quad\quad
\int_{\T^k}|\widehat{f}(\A)|^2 d\A=\sum_{m\in\Z^k}|f(m)|^2.\]

It is then easy to verify that from inequality (\ref{fB1}) we immediately obtain the estimate
\be\label{fB2}
\int_{\T^k}|\widehat{f_B}(\A)|^2|\widehat{1_S}(\A)|\,d\A\geq\frac{1}{4}\,\D|B||S|
\ee
where we recognize 
\be\label{weyl sum}
\widehat{1_S}(\A)=\sum_{d=1}^{\VE M} e^{-2\pi i (\A_1 d+\A_2 d^2+\cdots+\A_k d^k)},
\ee
as a classical Weyl sum. 

\subsection{Estimates for Weyl sums}

Since $\VE=(10k)^{-1}$ is fixed it is clear that whenever $|\A_j|\ll M^{-j}$ there can be no cancellation in the Weyl sum (\ref{weyl sum}), in fact the same is also true when each $\A_j$ is close to a rational with \emph{small} denominator (in other words there is no cancellation over sums in residue classes modulo $q$). 

We now state a precise formulation of the well-known fact that this is indeed the only obstruction to cancellation.
Let $\eta>0$. We define
\be
\mathbf{M}_q=\mathbf{M}_q(\eta)=\left\{\A\in\T^k\,:\,\Bigl|\A_j-\frac{a_j}{q}\Bigr|\leq\frac{1}{\eta^{k} M^{j}}\ (1\leq j\leq k) \ \text{for some $a\in[1,q]^k$
}\right\}.
\ee

\begin{lem}\label{Weyl Estimates}
Let $\eta>0$ and $M\geq \eta^{-C}$ (with $C$ sufficiently large depending on $k$). 
\begin{itemize}
\item[(i)] \emph{(Minor box estimate)} 
If $\A\notin\mathbf{M}_q$ for any $1\leq q\leq\eta^{-k}$, then
\[|\widehat{1_S}(\A)|\leq C\eta |S|
.\]
\item[(ii)] \emph{(Major box estimate)}
If $\A\in\mathbf{M}_q$ for some $1\leq q\leq\eta^{-k}$, then
\[|\widehat{1_S}(\A)|\leq C q^{-1/k} |S| 
.\]
\end{itemize}
\end{lem}
The proof of this result is a straightforward (and presumably well-known) consequence of the standard estimates for Weyl sums, for the sake of completeness we include these arguments in Appendix A.

\section{The proof of Proposition \ref{KeyPropn}}\label{proof}

In the previous section we established that inequalities (\ref{random2}) and (\ref{B'}) would be immediate consequences of
$B$ not exhibiting either of the two properties described in Proposition \ref{KeyPropn}. We now present the two lemmas from which we will obtain our desired contradiction.

In both lemmas below we set $\eta=\D/8C$, where $C$ is the large constant in Lemma \ref{Weyl Estimates}.

\begin{lem}\label{L1}
Let $\eta=\D/8C$. 
If $B$ is neither random nor structured, in the sense outlined in Proposition \ref{KeyPropn},
then there exists $1\leq q\leq\eta^{-k}$ such that
\be
\frac{1}{\D|B|}\int_{\mathbf{M}_q}|\widehat{f_B}(\A)|^2 d\A\geq c\D^{k-1}.
\ee
\end{lem}

The second lemma is a precise quantitative formulation, in our setting, of the standard $L^2$ density increment lemma.

\begin{lem}\label{L2}
Let $\eta=\D/8C$ and $\sigma\leq \eta^{k-2}/8\pi$. If $B$ is \emph{regular}, in the sense that \[|B\cap \Lambda|\leq\D(1+\sigma)|\Lambda|\]
for all arithmetic grids $\Lambda\subseteq Q_M$ of the form (\ref{grid}) with $qL\geq\eta^2\sigma M$, then
\be
\frac{1}{\D|B|}\int_{\mathbf{M}_q}|\widehat{f_B}(\A)|^2 d\A\leq 12\sigma.
\ee
\end{lem}

We therefore obtain a contradiction if $\sigma\leq c\D^{k-1}$, proving Proposition \ref{KeyPropn}.

\subsection{Proof of Lemma \ref{L1}}
It follows from the minor box estimate of Lemma \ref{Weyl Estimates} and Plancherel's identity that
\[\int\limits_{\text{minor boxes}}|\widehat{f_B}(\A)|^2|\widehat{1_S}(\A)|\,d\A\leq C\eta|S||B|.\]
Therefore, if $\eta=\D/8C$, it follows from estimate (\ref{fB2}) and the major box estimate of Lemma \ref{Weyl Estimates} that
\[\sum_{q=1}^{\eta^{-k}}q^{-1/k}\int_{\mathbf{M}_q}|\widehat{f_B}(\A)|^2 d\A\geq \eta|B|.\]
It therefore follows that
\[\max_{1\leq q\leq\eta^{-k}}\int_{\mathbf{M}_q}|\widehat{f_B}(\A)|^2 d\A\geq \eta^k|B|\]
as required.
\qed

\subsection{Proof of Lemma \ref{L2}}
We fix $q$ and $L$ so that $qL=\eta^2\sigma M$ and define
\[\Lambda=\{-(\ell_1 q,\ell_2 q^2,\dots,\ell_k q^k)\,|\, 1\leq\ell_j \leq L^j\}.\]

\begin{claim}\label{C1} If $\A\in \mathbf{M}_q$, then $|\widehat{1_{\Lambda}}(\A)|\geq |\Lambda|/2$. 
\end{claim}
\begin{proof}[Proof of Claim \ref{C1}]
Since
\[\sum_{j=1}^k L^j\|q^j\A_j\|
\leq \sum_{j=1}^k(Lq)^j \eta^{-k}M^{-j}
=\eta^{-k}\sum_{j=1}^k(\eta^2\sigma)^j
\leq 2\eta^{-(k-2)}\sigma,\]
for all $\A\in \mathbf{M}_q$, where $\|\cdot\|$ denotes the distance to the nearest integer, it follows that
\begin{align*}
|\widehat{1_{\Lambda}}(\A)|
\geq |\Lambda|-\sum_{\ell_j=1}^{L^j}\bigl|e^{2\pi i (\ell_1 q\A_1+\cdots+\ell_k q^k\A_k)}-1\bigr|
\geq |\Lambda|\Bigl(1-2\pi\sum_{j=1}^k L^j\|q^j\A_j\|\Bigr)
\geq |\Lambda|/2,
\end{align*}
for all $\A\in \mathbf{M}_q$, provided $\sigma\leq\eta^{k-2}/8\pi$.
\end{proof}
Plancherel's identity (applied to the function $f_B*1_\Lambda$) and Claim \ref{C1} imply that 
\be
\frac{1}{\D|B|}\int_{\mathbf{M}_q}|\widehat{f_B}(\A)|^2 d\A\leq\frac{4}{\D|B||\Lambda|^2}\sum_{m\in\Z^k}|f_B*1_\Lambda(m)|^2.
\ee
The conclusion of Lemma \ref{L2} will therefore be an immediate consequence of the following.

\begin{claim}\label{C2} As a consequence of the assumptions in Lemma \ref{L2} if follows that
\[\sum_{m\in\Z^k}|f_B*1_\Lambda(m)|^2\leq 3\sigma\,\D|B||\Lambda|^2.\]
\end{claim}

\begin{proof}[Proof of Claim \ref{C2}]
We let
\[\mathcal{M}=\{m\in\Z^k\,|\,m-\Lambda\subseteq Q_M\}\]
\[\mathcal{E}=(Q_M+\Lambda)\setminus\mathcal{M}\]
and write
\[\sum_{m\in\Z^k}|f_B*1_\Lambda(m)|^2=\sum_{m\in\mathcal{M}}|f_B*1_\Lambda(m)|^2+\sum_{m\in\mathcal{E}}|f_B*1_\Lambda(m)|^2.\]

We note that since
\[f_B*1_\Lambda(m)=|B\cap(m-\Lambda)|-\D|Q_M\cap(m-\Lambda)|\]
it follows from our \emph{regularity} assumption on $B$ that if $m\in\mathcal{M}$, then
\[-\D|\Lambda|\leq f_B*1_\Lambda(m)\leq \D\sigma|\Lambda|,\]
while for $m\in\mathcal{E}$ we can only conclude that
\[|f_B*1_\Lambda(m)|\leq |\Lambda|.\]

Now since $f_B$ has mean value zero the convolution 
\[f_B*1_{\Lambda}(m)=\sum_n f_B(n)1_{\Lambda}(m-n)\]
also has mean value zero. Thus, using the fact that $|g|=2g_+-g$, where $g_+=\max\{g,0\}$ denotes the \emph{positive-part} function, and the trivial size estimate $|\mathcal{M}|\leq |Q_M|$, we can deduce that
\begin{align*}
\sum_{m\in\mathcal{M}}|f_B*1_\Lambda(m)|^2
&\leq 2\left(\sup_{m\in\mathcal{M}}|f_B*1_{\Lambda}(m)|\right)\sum_{m\in\mathcal{M}}(f_B*1_\Lambda)_+(m)\\
&\leq
2(\D|\Lambda|)(\D\sigma|\Lambda|)|\mathcal{M}|\\
&\leq 2\D^2\sigma|\Lambda|^2|Q_M|.
\end{align*}

We leave it to the reader to verify that
\[|\mathcal{E}|\leq \left((1+2\eta^2\sigma)^k-(1-2\eta^2\sigma)^k\right)|Q_M|\leq 8k\eta^2\sigma |Q_M|,\]
and hence
\[\sum_{m\in\mathcal{E}}|f_B*1_\Lambda(m)|^2\leq |\Lambda|^2|\mathcal{E}|\ll \frac{1}{2}\sum_{m\in\mathcal{M}}|f_B*1_\Lambda(m)|^2,\]
provided $8k\eta^2\ll \D^2$. 

This concludes the proof of Claim \ref{C2} and establishes Lemma \ref{L2}. 
\end{proof}

\appendix

\section{Weyl sum estimates} 

\subsection{Standard major and minor arc estimates}

Let $P(\A,d)=\A_1d+\cdots+\A_kd^k$. 

\begin{lem}[Weyl inequality]\label{Weyl}
If $|\A_k-a_k/q|\leq q^{-2}$ and $(a,q)=1$, then
\[\left|\sum_{d=1}^N e^{2\pi i P(\A,d)}\right|\leq C_{k,\E} N^{1+\E}\left(\frac{1}{q}+\frac{1}{N}+\frac{q}{N^k}\right)^{1/2^{k-1}}.\]
\end{lem}

This result is completely standard, see for example \cite{TenLectures}.
We now fix a sufficiently small $\mu=\mu(k)>0$ and define
\be
\mathbf{M}_{a/q}'=\left\{\A\in\T^k\,:\,\Bigl|\A_j-\frac{a_j}{q}\Bigr|\leq \frac{1}{N^{j-\mu}} \ (1\leq j\leq k)\right\}.
\ee
Successive  applications of Dirichlet's principle and the Weyl inequality, starting with the highest power $k$, gives the following qualitative estimate (a quantitative version of which can be found in Vinogradov \cite{V}).

\begin{propn}[Minor arc estimate I]\label{minorI}
If $\A\notin \mathbf{M}_{a/q}'$ for any $(a,q)=1$ with $1\leq q \leq N^\mu$, then
\be
\left|\sum_{d=1}^N e^{2\pi i P(\A,d)}\right|\leq C N^{1-\nu}.
\ee
for some $\nu=\nu(k,\mu)>0$.
\end{propn}

\begin{propn}[Major arc estimate]\label{major}
If $\A\in \mathbf{M}_{a/q}'$ for some $(a,q)=1$ with $1\leq q\leq N^\mu$, then
\be 
\left|\sum_{d=1}^N e^{2\pi i P(\A,d)}\right|\leq C Nq^{-1/k}\Bigl(1+\sum_{j=1}^k N^j|\A_j-a_j/q|\Bigr)^{-1/k}+O(N^{1/2}).
\ee
\end{propn}
\begin{proof}
It is straightforward to write 
\[\sum_{d=1}^N e^{2\pi i P(\A,d)}=q^{-1} S(a,q)v_N(\A-a/q)+O(N^{1/2})\]
where
\[S(a,q):=\sum_{r=0}^{q-1}e^{2\pi i P(a,r)/q}\quad\text{and}\quad v_N(\B):=\int_0^N e^{2\pi i P(\B,x)}dx.\]
The result then follows from the observation that 
\be
|S(a,q)|\leq C q^{1-1/k}
\ee
whenever $(a,q)=1$, which is a result of Hua (see for example \cite{Vaughan}), and
\be
|v_N(\B)|\leq C N \Bigl(1+\sum_{j=1}^k N^j|\B_j|\Bigr)^{-1/k}
\ee
which follows from van der Corput's lemma for oscillatory integrals (see for example \cite{BigS}) and rescaling.
\end{proof}

\subsection{Refinement of the major arcs}

Let $0<\eta\leq 1$ and
\be
\mathbf{M}_{a/q}=\mathbf{M}_{a/q}(\eta)=\left\{\A\in \T^k\,:\,\Bigl|\A_j-\frac{a_j}{q}\Bigr|\leq \frac{1}{\eta^{k}N^j} \ (1\leq j\leq k)\right\}.
\ee

Combining Propositions \ref{minorI} and \ref{major} we easily obtain the following result from which Lemma \ref{Weyl Estimates} is an immediate consequence.

\begin{propn}[Minor arc estimate II]\label{minor} 
If $\A\notin \mathbf{M}_{a/q}$ for any $(a,q)=1$ with $1\leq q \leq \eta^{-k}$, then 
\[\left|\sum_{d=1}^N e^{2\pi i P(\A,d)}\right|\leq C\eta N + O(N^{1-\nu}).\] 
\end{propn}

\begin{proof}
It follows from Proposition \ref{major} that on $\mathbf{M}_{a/q}'$ we have
\[\left|\sum_{d=1}^N e^{2\pi i P(\A,d)}\right|\leq C\eta N\]
provided $(a,q)=1$ and either 
\[\eta^{-k}\leq q\leq N^\mu\]
or there exists $j$ such that
\[\eta^{-k}N^{-j}\leq |\A_j-a_j/q|\leq N^{-j+\mu}.\qedhere\]
\end{proof}


\begin{thebibliography}{10}

\bibitem{BPPS}
{\sc A. Balog, J. Pelik\'an, J. Pintz, E. Szemer\'edi}, {\em Difference sets without $\kappa$-th powers},
Acta Math. Hungar. 65 (1994), 165-187.

\bibitem{BL}
{\sc V.~Bergelson and A.~Leibman}, {\em Polynomial extensions of van der {W}aerden's and {S}zemer\'edi's theorems}, J. Amer. Math. Soc., 9, No. 2 (1996), 725-753.

\bibitem{Furst}
{\sc H.~Furstenberg}, {\em Ergodic behavior of diagonal measures and a theorem of {S}zemer\'edi on arithmetic progressions},
  J. d'Analyse Math, 71 (1977), 204-256.

\bibitem{Gowers}
{\sc W.~T.~Gowers}, {\em A new proof of {S}zemer\'edi's theorem}, GAFA, 11 (2001), 465-588.

\bibitem{Green}
{\sc B. Green}, {\em On arithmetic structures in dense sets of integers}, Duke Math. Jour., 114, (2002) (2), 215-238.

\bibitem{HB}
{\sc D. R. Heath-Brown}, {\em Integer sets containing no arithmetic progressions},
J. London Math. Soc. (2) 35(3) (1987), 385-394.

\bibitem{Lucier}
{\sc J. Lucier}, {\em Intersective sets given by a polynomial},
   Acta Arith.  123  (2006),  no. 1, 57-95.

\bibitem{NL}
{\sc N. Lyall and \'A. Magyar}, {\em Polynomial configurations in difference sets},
   J. Num. Theory, v. 129/2, pp. 439-450, 2009.

\bibitem{TenLectures}
{\sc H. L. Montgomery}, {\em Ten Lectures on the Interface Between Analytic Number Theory and Harmonic Analysis}, CBMS Regional Conference Series in Mathematics, 84.

\bibitem{PSS}
{\sc J. Pintz, W. L. Steiger, E. Szemer\'edi}, {\em On sets of natural numbers whose difference set contains no squares},
J. London Math. Soc. 37 (1988), 219-231.

\bibitem{Ruzsa}
{\sc I. Z. Ruzsa}, {\em Difference sets without squares},
Period. Math. Hungar. 15 (1984), 205-209.

\bibitem{RandS}
{\sc I. Z. Ruzsa and T. Sanders}, {\em Difference sets and the primes},
  preprint.

\bibitem{Sarkozy}
{\sc A.~S\'arz\"ozy}, {\em On difference sets of sequences of integers III},
  Acta Math. Acad. Sci. Hungar., 31 (1978), 355-386.

\bibitem{BigS}
{\sc E. M. Stein}, {\em Harmonic Analysis:
  Real--Variable Methods, Orthogonality, and Oscillatory Integrals}, Princeton
  Univ. Press, Princeton, 1993.  

\bibitem{SzRoth}
{\sc E.~Szemer\'edi}, {\em Integer sets containing no arithmetic progressions},
  Acta Math. Hungar., 56(1-2) (1990), 155-158.

\bibitem{Vaughan}
{\sc R. C. Vaughan}, {\em The Hardy--Littlewood Method}, 2nd ed., Cambridge Univ. Press, Cambridge, 1997.
  
\bibitem{V}
{\sc I. M. Vinogradov}, {\em The Method of Trigonometrical Sums in the Theory of Numbers}, Interscience, New York, 1954.

\end{thebibliography}
\end{document}